\documentclass[a4paper]{amsart}
\usepackage{amsthm,amsfonts,amsmath,amssymb}
\usepackage[abs]{overpic}
\usepackage{comment} %comment out
\usepackage{color} %color

\newtheorem{theorem}{Theorem}[section]
\newtheorem{lemma}[theorem]{Lemma}
\newtheorem{proposition}[theorem]{Proposition}
\newtheorem{corollary}[theorem]{Corollary}
\newtheorem{definition}[theorem]{Definition}
\newtheorem{remark}[theorem]{Remark}
\newtheorem{claim}{Claim}

\newtheorem{example}[theorem]{Example}

\newcommand{\omu}{\overline{\mu}}

\begin{document}

\title{The Milnor invariants of clover links}
\author{Kodai Wada}
\author{Akira Yasuhara}
\thanks{ 
The second author is partially supported by a JSPS Grant-in-Aid for Scientific Research (C) 
($\#$26400081). }

\address{Department of Mathematics, School of Education, Waseda University, Nishi-Waseda 1-6-1, Shinjuku-ku, Tokyo, 169-8050, Japan}
\address{Department of Mathematics, Tokyo Gaugei University, 4-1-1 Nukuikita-machi, Koganei-shi, Tokyo, 184-8501, Japan}
\email{k.wada@akane.waseda.jp\\yasuhara@u-gakugei.ac.jp}
\date{\today}
\maketitle

%%%%%%%%%%%%%%%%%%%%%%%%%%%%%%%%%%%%%%%%%%%%%%%%%%%%%%%%%%%%%%%%%%%%%%%%%%%%%%%%%%%%%%%%%%%%%%%%%%
\begin{abstract}
J.P. Levine introduced  a clover link to investigate the indeterminacy of the Milnor invariants of a link. 
It is shown that for a clover link, the Milnor numbers  of length at most $2k+1$ are well-defined 
if those of length at most $k$ vanish, 
and that the Milnor numbers of length at least $2k+2$ are not well-defined  
if those of length $k+1$ survive. 
For a clover link $c$ with the Milnor numbers of length at most $k$ vanishing,  
we show that the Milnor number $\mu_c(I)$ for a sequence $I$  is well-defined 
up to the greatest common devisor of $\mu_{c}(J)'s$, where 
$J$ is a subsequence of $I$ obtained
by removing at least $k+1$ indices.  
Moreover, if $I$ is  a non-repeated sequence with length $2k+2$, 
the possible range of $\mu_c(I)$ is given explicitly.  
As an application, we give an edge-homotopy classification of $4$-clover links.
\end{abstract}

%%%%%%%%%%%%%%%%%%%%%%%%%%%%%%%%%%%%%%%%%%%%%%%%%%%%%%%%%%%%%%%%%%%%%%%%%%%%%%%%%%%%%%%%%%%%%%%%%%
\section{Introduction}
The {\em Milnor invariant} introduced by J. Milnor~\cite{M}, \cite{M2}. 
For an oriented ordered $n$-component link $L$ in the $3$-sphere $S^3$ with 
peripheral information,   
the {\it Milnor number $\mu_{L}(I)$}, which is an integer, is 
specified by a finite sequence $I$ in $\{ 1,2,\ldots,n\}$.
The Milnor $\omu$-invariant $\omu_{L}(I)$ is the residue class of $\mu_{L}(I)$ modulo
the greatest common devisor of $\mu_{L}(J)$'s, where $J$ is obtained from proper subsequence of $I$ by permuting cyclicly.
The length of the sequence $I$ is called the {\it length} of $\omu_{L}(I)$ and denoted by $|I|$.
His original definition of the Milnor invariant eliminates the indeterminacy of the possible variations of the Milnor numbers caused by different choices of peripheral elements.

In~\cite{L}, J. P. Levine examined the Milnor invariants from the point of view of {\it based} links, in order to understand the indeterminacy.
A based link is a link for which some peripheral information is specified,
i.e., meridians (the weakest) or both meridians and longitudes (the strongest).
It is known that these invariants are completely well-defined for the strongest form of basing (disk links~\cite{Le} or string links~\cite{HL}).
As basing only slightly stronger than the specification of meridians, he introduced an {\it n-clover link} which is an embedded graph consisting of $n$ loops, each loop connected to a vertex by an edge in $S^3$.
The Milnor number of a clover link is defined up to flatly isotopy
and it is shown that
those of length $\leq k$ vanish implies those of length $\leq 2k+1$ are completely non-indeterminate~\cite{L}.

The first author~\cite{W} redefined the Milnor number of a clover link up to ambient isotopy as follows.
Given an $n$-clover link $c$, we construct an $n$-component {\it bottom tangle} $\gamma(F_{c})$ by using a {\em disk/band surface} $F_{c}$ of $c$.
In \cite{L}, Levine defined the Milnor number of a bottom tangle.
Therefore we define the {\it Milnor number $\mu_{c}$ of an $n$-clover link $c$} to be the Milnor number $\mu_{\gamma(F_c)}$.
(In \cite{L}, a bottom tangle is called a string link.
The name \lq bottom tangle\rq ~follows K. Habiro \cite{H2}.)
In \cite{W}, it is shown that the same result as Levine~\cite{L} holds while 
there are infinitely many choices of $\gamma(F_c)$ for $c$. 

Unfortunately, the Milnor numbers of length at least $2k+2$ are not well-defined  
if those of length $k+1$ survive. 
In this paper, we show that the Milnor number $\mu_c(I)$ of  a sequence $I$  
modulo $\delta_c^k(I)$ is well-defined 
if the Milnor numbers of length~$\leq k$ vanish,   
where $\delta_c^k(I)$ is the greatest common devisor of $\mu_{c}(J)'s$, where  
$J$ is range to over all subsequences of $I$ obtained by removing at least $k+1$ indices. 
In fact, we have the following theorem. 
\begin{theorem}
\label{mainthmmodulo}
Let $c$ be an $n$-clover link and $l_c$ a link which is the disjoint union of loops of $c$.
If the Milnor numbers of $l_c$ for sequences with length $\leq k$ vanish,
then the residue class of $\mu_c(I)$ modulo $\delta_c^k(I)$ is an invariant for any sequence $I$.
\end{theorem}

\begin{remark}
In contrast to the $\omu$-invariant for a link in $S^3$, we do not need to take cyclic permutation for 
getting $\delta_c^k(I)$. 
\end{remark}

It is an important property that
for non-repeated sequences,
the Milnor invariants of links are link-homotopy invariants~\cite{M}.
The first author showed that the Milnor numbers for any non-repeated sequence 
with length $\leq 3$ give an {\em edge-homotopy} classification of $3$-clover links~\cite{W},
where edge-homotopy \cite{T} is an equivalence relation,
which is a generalization of link-homotopy,
generated by crossing changes on the same spatial edge. 
We also discuss giving an edge-homotopy classification of $4$-clover links.
The Milnor numbers for non-repeated sequences with 
length 4 could be useful to have an edge-homotopy  
classification of 4-clover links. 
But they are not well-defined in general.  
Hence we consider the set of all Milnor numbers of length 4 for all disk/band 
surfaces of $c$. 
More generally, we define the following set; 
\[H_c(2k+2, j)=\displaystyle
\left\{\left. \sum_{S\in \mathcal{S}_j^{2k+1}}\mu_{\gamma(F_c)}(Sj)X_S~\right|~F_c:\text{a disk/band surface of c}\right\}\]
for each integer $j~(1\leq j\leq n)$,   where $\mathcal{S}_j^{2k+1}$ is the set of 
length-$(2k+1)$ non-repeated sequences  without containing $j$
and for a sequence $S=i_1i_2\ldots i_{2k+1}$, $X_S=X_{i_1}X_{i_2}\cdots X_{i_{2k+1}}$ 
is a monomial in non-commutative variables $X_1,...,X_n$.
Since $H_c(2k+2, j)$ consists of the Milnor numbers  
$\mu_{\gamma(F_c)}(Sj)~(S\in\mathcal{S}_j^{2k+1})$ for all  
disk/band surfaces of $c$, it is an invariant of $c$. 
While it seems too big to handle  $H_c(2k+2, j)$, we have the following theorem. 
\begin{theorem}
\label{mainthm2k+2}
Let $c$ be an $n$-clover link and $F_c$ a disk/band surface of $c$. 
If the Milnor numbers of $l_c$ for non-repeated sequences with length $\leq k$  vanish,
then we have the following: 
\[\begin{array}{l}
H_c(2k+2, j)=\\
\displaystyle
\hspace*{-.5em}
\left\{\sum_{\substack{|J|=|I|=k\\JIl\in \mathcal{S}_j^{2k+1}}}
\hspace{-1em}\mu_{\gamma(F_c)}(Jj)\mu_{\gamma(F_c)}(Il)
\left(\sum_{i_s\in\{J\}}\hspace{-.5em}m_{li_s}
X_{J_{<s}}(
X_{i_sIl}
-X_{i_slI}
-X_{Ili_s}
+X_{lIi_s}
)X_{J_{s<}}\right. \right. \\
\hspace*{.5em}+\displaystyle \left.\left.\Bigg.
m_{lj}(X_{IlJ}-X_{lIJ}-X_{JIl}+X_{JlI}) \Biggr) +
\sum_{S\in \mathcal{S}_j^{2k+1}}\hspace{-.5em}\mu_{\gamma(F_c)}(Sj)X_S
~\right|~ m_{pq}=m_{qp}\in {\mathbb Z}\right\},
\end{array}\]
where for a sequence $J=i_1\ldots i_m$,
$\{J\}=\{i_1,\ldots,i_m\}$ and $J_{<s}$ {\rm (}resp. $J_{s<}${\rm )} is a subsequence $i_1\ldots i_{s-1}$ {\rm (}resp. $i_{s+1}\ldots i_m${\rm )} of $J$ for 
$1\leq s\leq m$, and $X_{J_{<1}}$ and $X_{J_{m<}}$ are 
defined to be $1$.
\end{theorem}

\noindent
This theorem implies that  
the set $H_c(2k+2, j)$ is obtained from the Milnor numbers of $\gamma(F_c)$ 
for any single disk/band surface $F_c$ of $c$, that is, 
$H_c(2k+2, j)$ is specified explicitly. 

By the following corollary we have that $H_c(2k+2, j)$ is 
not only an invariant of $c$ but also an edge-homotopy invariant.

\begin{corollary}\label{EH-inv}
For a clover link $c$, if the Milnor numbers of $l_c$ for non-repeated sequences
 with length $\leq k$ vanish, then 
$H_c(2k+2, j)$ is an edge-homotopy invariant of $c$.
\end{corollary}

It is the definition that the Milnor numbers of length $1$ are zero.
If $k=1$, then the theorem above holds without the condition.
The following example follows directly from Theorem~\ref{mainthm2k+2}.

\begin{example}
\label{ex4-clover}
For a $4$-clover link $c$ and a disk/band surface $F_c$ of $c$, we have 
\[\begin{array}{ll}
H_c(4, 4)=&
\Biggl\{\vspace{0.5em}\left(
\begin{array}{l}
\mu_{\gamma(F_c)}(14)\mu_{\gamma(F_c)}(23)(m_{13}-m_{34}-m_{12}+m_{24})\\
+\mu_{\gamma(F_c)}(12)\mu_{\gamma(F_c)}(34)(m_{24}-m_{23}-m_{14}+m_{13})
\end{array}\right)X_{123}\\
\hspace{1em}
\vspace{0.5em}&+\left(
\begin{array}{l}
\mu_{\gamma(F_c)}(14)\mu_{\gamma(F_c)}(23)(m_{34}-m_{13}-m_{24}+m_{12})\\
+\mu_{\gamma(F_c)}(13)\mu_{\gamma(F_c)}(24)(m_{34}-m_{23}-m_{14}+m_{12})
\end{array}\right)X_{132}\\
\hspace{1em}
\vspace{0.5em}&+\left(
\begin{array}{l}
\mu_{\gamma(F_c)}(13)\mu_{\gamma(F_c)}(24)(m_{23}-m_{34}-m_{12}+m_{14})\\
+\mu_{\gamma(F_c)}(12)\mu_{\gamma(F_c)}(34)(m_{23}-m_{24}-m_{13}+m_{14})
\end{array}\right)X_{213}
\\
\hspace{1em}
\vspace{0.5em}&+\left(
\begin{array}{l}
\mu_{\gamma(F_c)}(14)\mu_{\gamma(F_c)}(23)(m_{34}-m_{13}-m_{24}+m_{12})\\
+\mu_{\gamma(F_c)}(13)\mu_{\gamma(F_c)}(24)(m_{34}-m_{23}-m_{14}+m_{12})
\end{array}\right)X_{231}
\\
\hspace{1em}
\vspace{0.5em}&+\left(
\begin{array}{l}
\mu_{\gamma(F_c)}(13)\mu_{\gamma(F_c)}(24)(m_{23}-m_{34}-m_{12}+m_{14})\\
+\mu_{\gamma(F_c)}(12)\mu_{\gamma(F_c)}(34)(m_{23}-m_{24}-m_{13}+m_{14})
\end{array}\right)X_{312}
\\
\hspace{1em}
\vspace{0.5em}&+\left(
\begin{array}{l}
\mu_{\gamma(F_c)}(14)\mu_{\gamma(F_c)}(23)(m_{13}-m_{34}-m_{12}+m_{24})\\
+\mu_{\gamma(F_c)}(12)\mu_{\gamma(F_c)}(34)(m_{24}-m_{23}-m_{14}+m_{13})
\end{array}\right)X_{321}\\
\hspace{1em}
\vspace{0.5em}&+\displaystyle\sum_{S\in \mathcal{S}_4^3}\mu_{\gamma(F_c)}(S4)X_S~\Bigg|~m_{pq}\in{\mathbb Z}\Biggr\}.
\end{array}\]
\end{example}

This together with the theorem below gives us an edge-homotopy  classification of 
4-clover links, see Remark~\ref{rem: class}.
\begin{theorem}
\label{thm4-clover}
Let $c$ and $c'$ be $4$-clover links.
They are edge-homotopic  if and only if 
$H_c(4, 4)\cap H_{c'}(4, 4)\neq\emptyset$
and $\mu_c(I)=\mu_{c'}(I)$ for any non-repeated sequence $I$ with $|I|\leq 3$.
\end{theorem}

\begin{remark}\label{rem: class}
By Example~{\rm \ref{ex4-clover}}, 
we are able to determine  whether
$H_c(4, 4)\cap H_{c'}(4, 4)$ is empty or not. 
Hence by combining Example~{\rm \ref{ex4-clover}} and Theorem~{\rm \ref{thm4-clover}},
we obtain an edge-homotopy classification of $4$-clover links.
\end{remark}

%%%%%%%%%%%%%%%%%%%%%%%%%%%%%%%%%%%%%%%%%%%%%%%%%%%%%%%%%%%%%%%%%%%%%%%%%%%%%%%%%%%%%%%%%%%%%%%%%%%%%%%%%%%%%%%%%%%%%%%%%%%%%%%%%%%%%%%%%%%%%%%%%%%%%%%%%%%%%%%%%%%%%%%%%%%%%%%%%%%%%%%%%%%%%%%%%%%%%%%%%%%%%%%%%%%%%%%%%%%%%%%%%%%%%%%%%%%%%%%%%%%%%%%%%%%%%%%%%%%%%%%%%%%%%%
\section{The Milnor numbers of  clover links}
In this section we define the Milnor numbers for clover links.
\subsection{A construction of bottom tangles}

An $n$-component {\it tangle} is a properly embedded disjoint union of $n$ arcs in the $3$-cube $[0,1]^{3}$.
An $n$-component {\it bottom tangle} $\gamma=\gamma_{1}\cup \gamma_{2}\cup \cdots \cup \gamma_{n}$ defined by Levine~\cite{L} is a tangle with $\partial \gamma_{i}=\{ (\frac{2i-1}{2n+1},\frac{1}{2},0),(\frac{2i}{2n+1},\frac{1}{2},0) \}$ $\subset$ $\partial [0,1]^{3}$ 
 for each  $i\ (=1,2,\ldots,n)$.

A {\it spatial graph} is an embedded graph in $S^3$.
Let $C_n$ be a graph consisting of $n$ oriented loops $e_{1},e_{2},\ldots ,e_{n}$, each loop $e_{i}$ connected to a vertex $v$ by an edge $f_{i}\ (i=1,2,\ldots,n)$ , see Figure~\ref{Cn-graph}.
An {\it $n$-clover link} in $S^3$ is a spatial graph of $C_n$~\cite{L}. 
The each part of a clover link corresponding to $e_{i}$, $f_{i}$ and $v$ of $C_n$ are 
called the {\it leaf}, {\it stem} and {\it root}, denoted by the same notations respectively.

\begin{figure}[htpb]
 \begin{center}
\vspace{-3mm}
  \begin{overpic}[width=5cm]{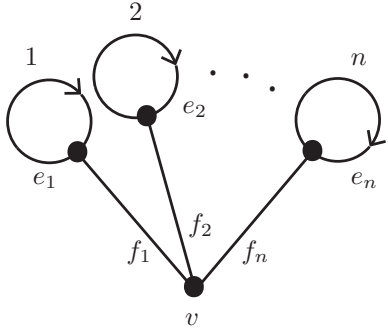}
     \put(7,108){1}
     \put(46,125){2}
     \put(129,108){$n$}
     \put(10,63){$e_{1}$}
     \put(66,90){$e_{2}$}
     \put(129,63){$e_{n}$}
     \put(45,35){$f_{1}$}
     \put(68,45){$f_{2}$}
     \put(88,35){$f_{n}$}
     \put(67,10){$v$}
  \end{overpic}
\vspace{-3mm}
  \caption{the graph $C_{n}$}
  \label{Cn-graph}
 \end{center}
\end{figure}

L. Kauffman, J. Simon, K. Wolcott and P. Zhao \cite{KSWZ} defined disk/band surfaces of spatial graphs.
For a spatial graph $\Gamma$, a {\it disk/band surface} $F_{\Gamma}$ of $\Gamma$ is a compact, oriented surface in $S^3$ such that $\Gamma$ is a deformation retract of $F_{\Gamma}$ contained in the interior of $F_{\Gamma}$. 
Note that any disk/band surface of a spatial graph is ambient isotopic to a surface constructed by putting a disk at each vertex of the spatial graph, connecting the disks with bands along the spatial edges.
We remark that for a spatial graph, there are infinitely many disk/band surfaces up to ambient isotopy.

Given an $n$-clover link, we construct an $n$-component bottom tangle using a disk/band surface of the clover link as follows:
\begin{itemize}
\item[(1)] For an $n$-clover link $c$, let $F_{c}$ be a disk/band surface of $c$ and let $D$ be a disk which contains the root.
From now on, we may assume that the intersection $D\cap \displaystyle\bigcup_{i=1}^{n}f_{i}$ and 
orientations of the disks are as illustrated in Figure~\ref{diskbandnotorikata}. 

\begin{figure}[htpb]
 \begin{center}
  \begin{overpic}[width=8cm]{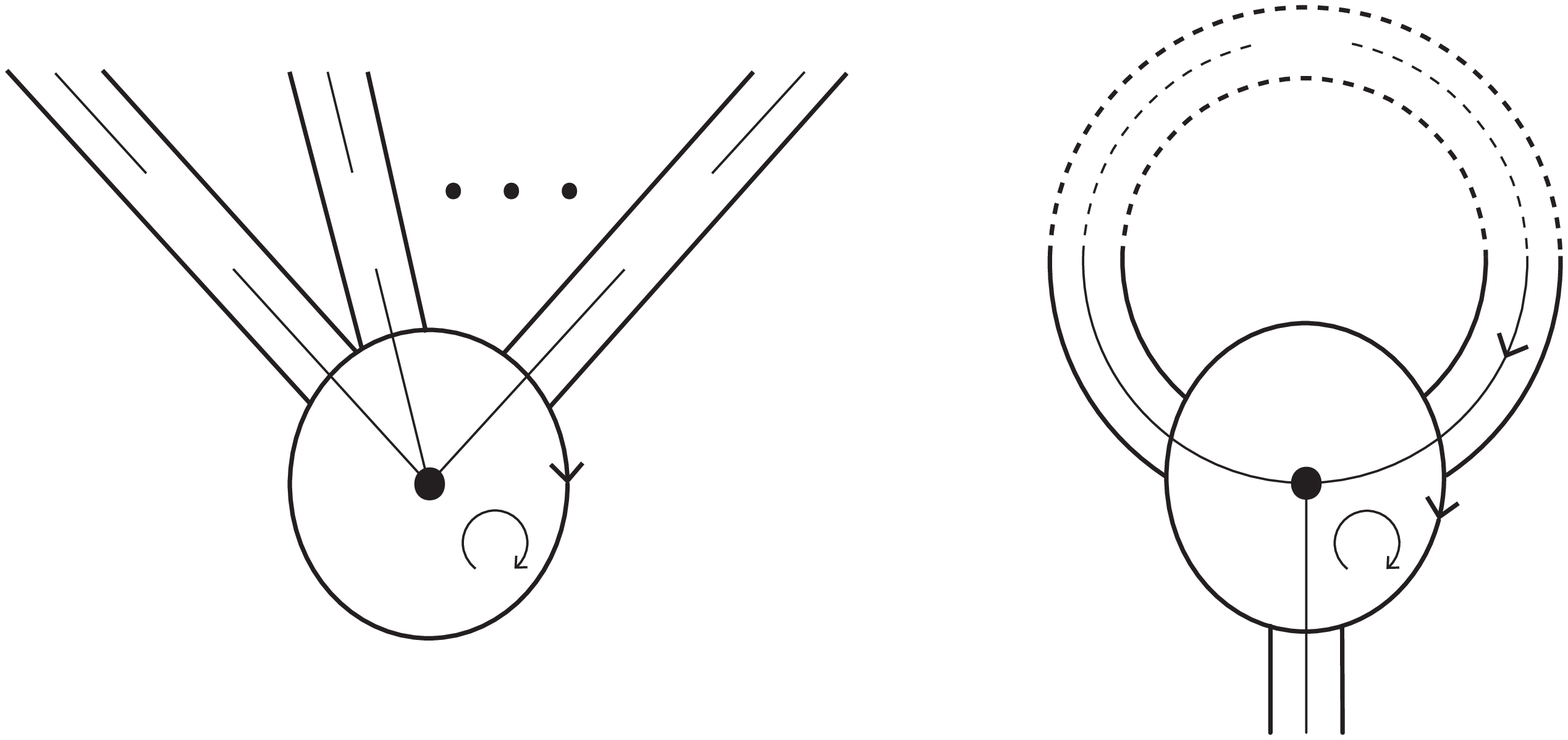}
     \put(23,73){$f_{1}$}
     \put(48,73){$f_{2}$}
     \put(92,73){$f_{n}$}
     \put(58,0){$D$}
     \put(187,99){$e_{i}$}
  \end{overpic}
  \caption{}
  \label{diskbandnotorikata}
 \end{center}
\end{figure}

\item[(2)] Let $N(D)$ be the regular neighborhood of $D$ and $\mathring{N(D)}$ the interior of $N(D)$.
Since $S^{3}\setminus \mathring{N(D)}$ is homeomorphic to the $3$-ball, $F_{c}\setminus \mathring{N(D)}$ can be seen as a disjoint union of surfaces in the $3$-ball.
Hence $\partial F_{c}\setminus \mathring{N(D)}$ is a disjoint union of $n$-arcs and $n$-circles $\displaystyle\bigcup_{i=1}^{n}S_{i}^{1}$ in the $3$-ball.
\item[(3)] Since the $3$-ball is homeomorphic to $[0,1]^{3}$, we obtain an oriented ordered $n$-component bottom tangle $\gamma(F_{c})$ from $(\partial F_{c}\setminus \mathring{N(D)}) \setminus \displaystyle\bigcup_{i=1}^{n}S_{i}^{1}$ as illustrated in (3) and (4) of Figure~\ref{bottomtanglenokosei}.
We call $\gamma(F_c)$ an {\it $n$-component  bottom tangle obtained from $F_{c}$}.
\end{itemize}

\begin{figure}[htbp]
  \begin{center}
  \vspace{-3mm}
  \begin{overpic}[width=10cm]{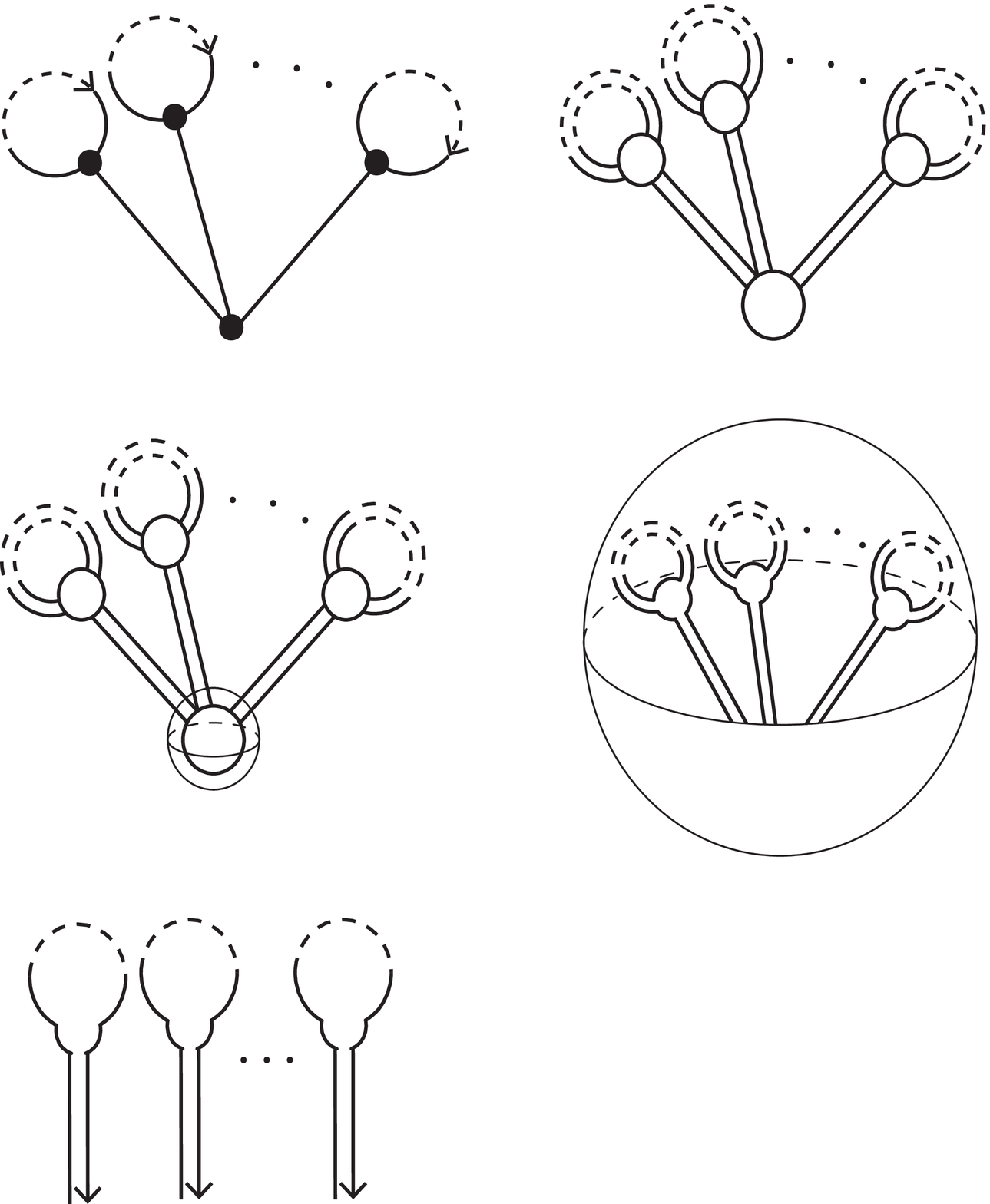}
   \put(1,325){1}
   \put(39,345){2}
   \put(124,328){$n$}
   \put(-25,285){$c$ :}
   \put(145,285){(1)}
   \put(235,253){$D$}
   \put(-25,155){(2)}
   \put(79,129){$N(D)$}
   \put(145,155){(3)}
   \put(200,85){$\partial F_{c}\setminus \mathring{N(D)}$}
   \put(-25,38){(4)}
   \put(20,85){1}
   \put(52,85){2}
   \put(96,85){$n$}
   \put(60,-5){$\gamma(F_c)$}
\end{overpic}
    \caption{A method for obtaining a bottom tangle from a disk-band surface of a clover link}
    \label{bottomtanglenokosei}
  \end{center}
\end{figure}

While there are infinitely many disk/band surfaces of a clover link
which satisfy the condition (1) above,
they are related by certain local moves as follows.

\begin{lemma}\cite[Proposition 2.5]{W}
\label{lemmaPB-move}
For an $n$-clover link $c$, any two disk/band surfaces $F_{c}$ and $F'_{c}$ are transformed into each other by adding full-twists to bands {\rm (}Figure~{\rm \ref{PB-move} (a))} and a single move illustrated in 
Figure~{\rm \ref{PB-move} (b)}.
\end{lemma}
\begin{figure}[htpb]
 \begin{center}
 %\vspace{-3mm}
  \begin{overpic}[width=8cm]{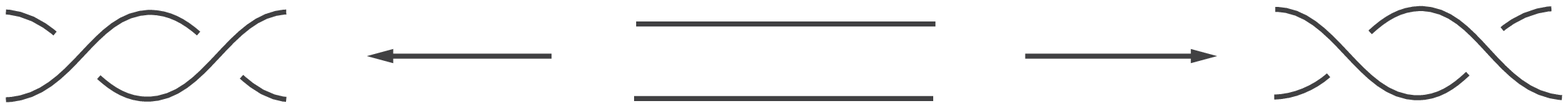}
     \put(-20,5){(a)}
     \put(103,3){band}
     \put(47,11){full-twist}
     \put(143,11){full-twist}
  \end{overpic}
 \end{center}
\end{figure}

\begin{figure}[htbp]
  \begin{center}
  \vspace{-5mm}
   \begin{overpic}[width=8cm]{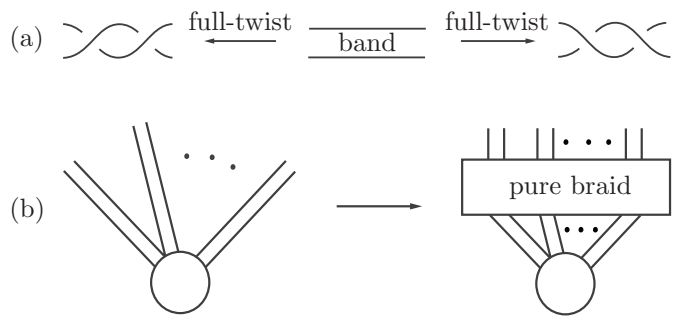}
    \put(-20,37){(b)}
    \put(167,46){pure braid}
    \end{overpic}
    \caption{Two local moves of disk/band surfaces}
    \label{PB-move}
  \end{center}
\end{figure}
%%%%%%%%%%%%%%%%%%%%%%%%%%%%%%%%%%%%%%%%%%%%%%%%%%%%%%%%%%%%%%%%%%%%%%%%%%%%%%%%%%%%%%%%%%%%%%%%%%%%%%%%%%%%%%%%%%%%%%%%%%%%%%%%%%
\subsection{Milnor invariants}
Let us briefly recall from \cite{L} the definition of the Milnor number of a bottom tangle.
Let $\gamma=\gamma_{1}\cup\gamma_{2}\cup\cdots\cup\gamma_{n}$ be an oriented ordered $n$-component bottom tangle in $[0,1]^{3}$ 
with $\partial \gamma_{i}=\{ (\frac{2i-1}{2n+1},\frac{1}{2},0),(\frac{2i}{2n+1},\frac{1}{2},0) \}$ $\subset$ $\partial [0,1]^{3}$ for each  $i\ (=1,2,\ldots,n)$.
%, and let $p_i=(\frac{2i-1}{2n+1},\frac{1}{2},0),~q_i=(\frac{2i}{2n+1},\frac{1}{2},0)$.
Let $G$ be the fundamental group of $[0,1]^{3}\setminus \gamma$ with a base point $p=(\frac{1}{2},0,0)$ and $G_{q}$ the $q$th lower central subgroup of $G$.
Let $\alpha_i$ and $\lambda_i$ be the $i$th meridian and $i$th longitude of $\gamma$ respectively as illustrated in Figures~\ref{meridian}.
We assume that $\lambda_{i}$ is trivial in $G/G_{2}$.
Since the quotient group $G/G_{q}$ is generated by $\alpha_{1},\alpha_{2},\ldots,\alpha_{n}$~\cite{Ch}, $\gamma$ is represented by $\alpha_{1},\alpha_{2},\ldots,\alpha_{n}$ modulo $G_{q}$.

\begin{figure}[htbp]
\begin{minipage}{0.45\hsize}
  \begin{center}
    \begin{overpic}[width=5cm]{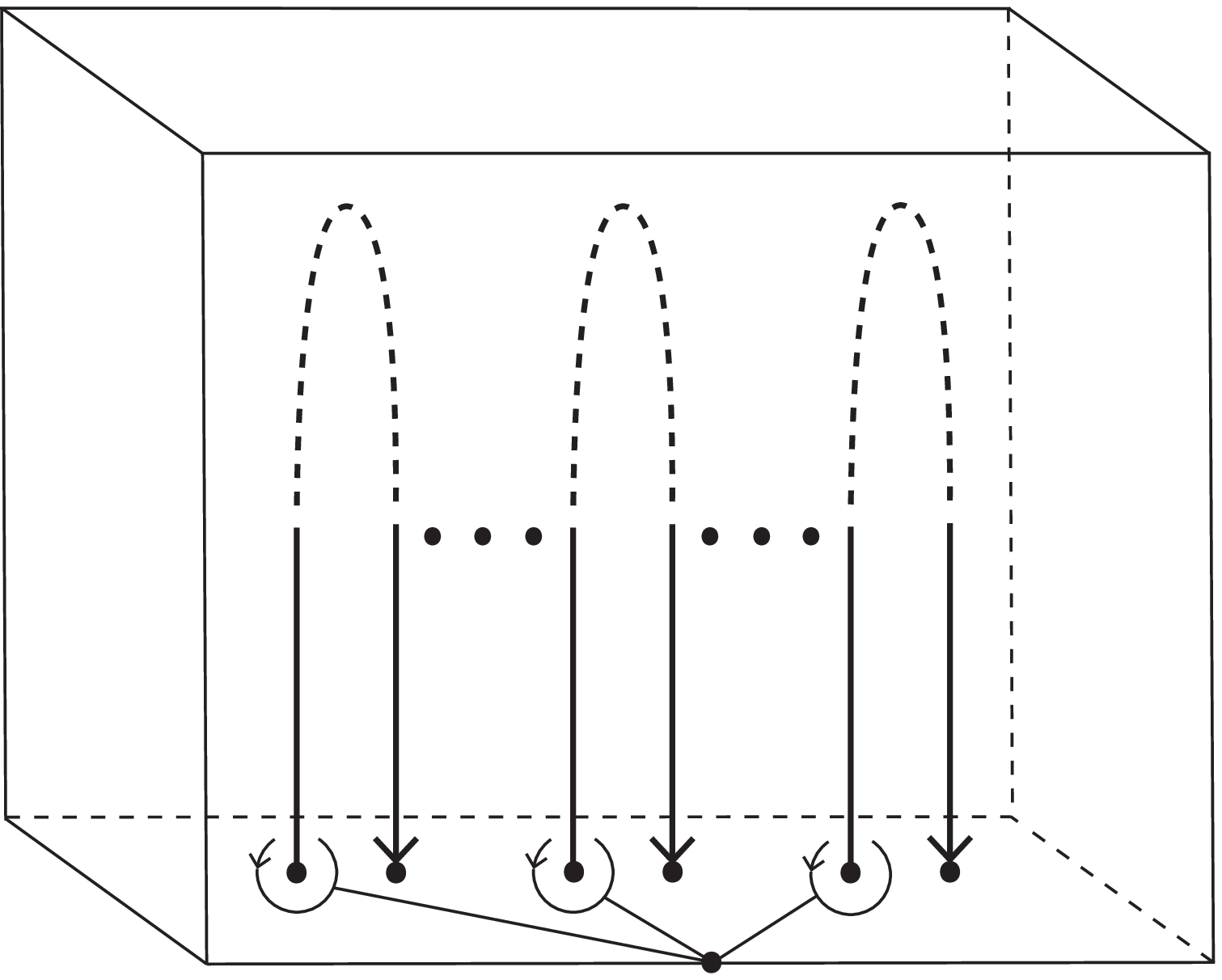}
       \put(26,10){$\alpha_{1}$}
       \put(52,18){$\alpha_{i}$}
       \put(102,10){$\alpha_{n}$}
       \put(81,0){$p$}
     \end{overpic}
  \end{center}
\end{minipage}
\begin{minipage}{0.45\hsize}
 \begin{center}
    \begin{overpic}[width=5cm]{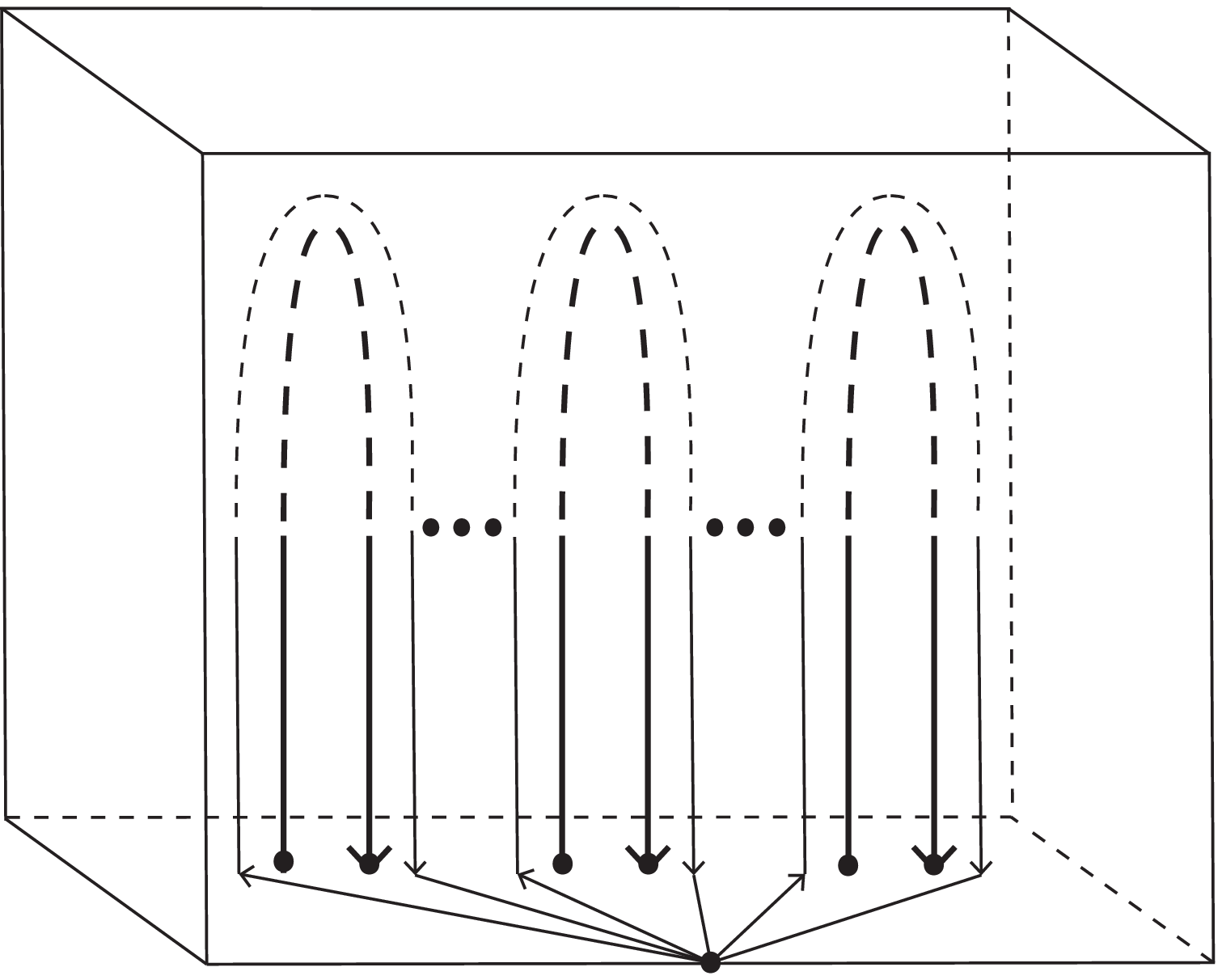}
       \put(25,10){$\lambda_{1}$}
       \put(51,28){$\lambda_{i}$}
       \put(112,10){$\lambda_{n}$}
       \put(81,0){$p$}
     \end{overpic}
  \end{center}
\end{minipage}
\caption{meridians and longitudes}
\label{meridian}
\end{figure}

We consider the {\it Magnus expansion} of $\lambda_{j}$.
The Magnus expansion is a homomorphism (denoted $E$) from a free group $\langle \alpha_{1}, \alpha_{2}, \ldots , \alpha_{n}\rangle$ to the formal power series ring in non-commutative variables $X_{1},X_{2},\ldots,X_{n}$ with integer coefficients defined as follows.
$E(\alpha_{i})=1+X_{i},\ E(\alpha_{i}^{-1})=1-X_{i}+X_{i}^{2}-X_{i}^{3}+\cdots \ (i=1,2,\ldots , n)$.

For a sequence $I=i_{1}i_{2}\ldots i_{k-1}j\ (i_{m}\in \{1,2,\ldots,n\}, k\leq q)$, we define the {\it Milnor number} $\mu_{\gamma}(I)$ to be the coefficient of 
$X_{i_{1}i_{2}\cdots i_{k-1}}$ in $E(\lambda_{j})$ (we define $\mu_{\gamma}(j)=0$),
which is an invariant \cite{L}.
(In \cite{L}, the set of $\lambda_{i}$'s, without taking the Magnus expansion, is called the {\it Milnor's $\omu$-invariant}.)
For a bottom tangle $\gamma=\gamma_{1}\cup\gamma_{2}\cup\cdots\cup\gamma_{n}$, an oriented link $L(\gamma)=L_{1}\cup L_{2}\cup\cdots\cup L_{n}$ in $S^{3}$ can be defined by $L_{i}=\gamma_{i}\cup a_{i}$, where $a_{i}$ is a line segment connecting $(\frac{2i-1}{2n+1},\frac{1}{2},0)$ and $(\frac{2i}{2n+1},\frac{1}{2},0)$, see Figure~\ref{closure}.
\begin{figure}[htbp]
  \begin{center}
  \begin{overpic}[width=5cm]{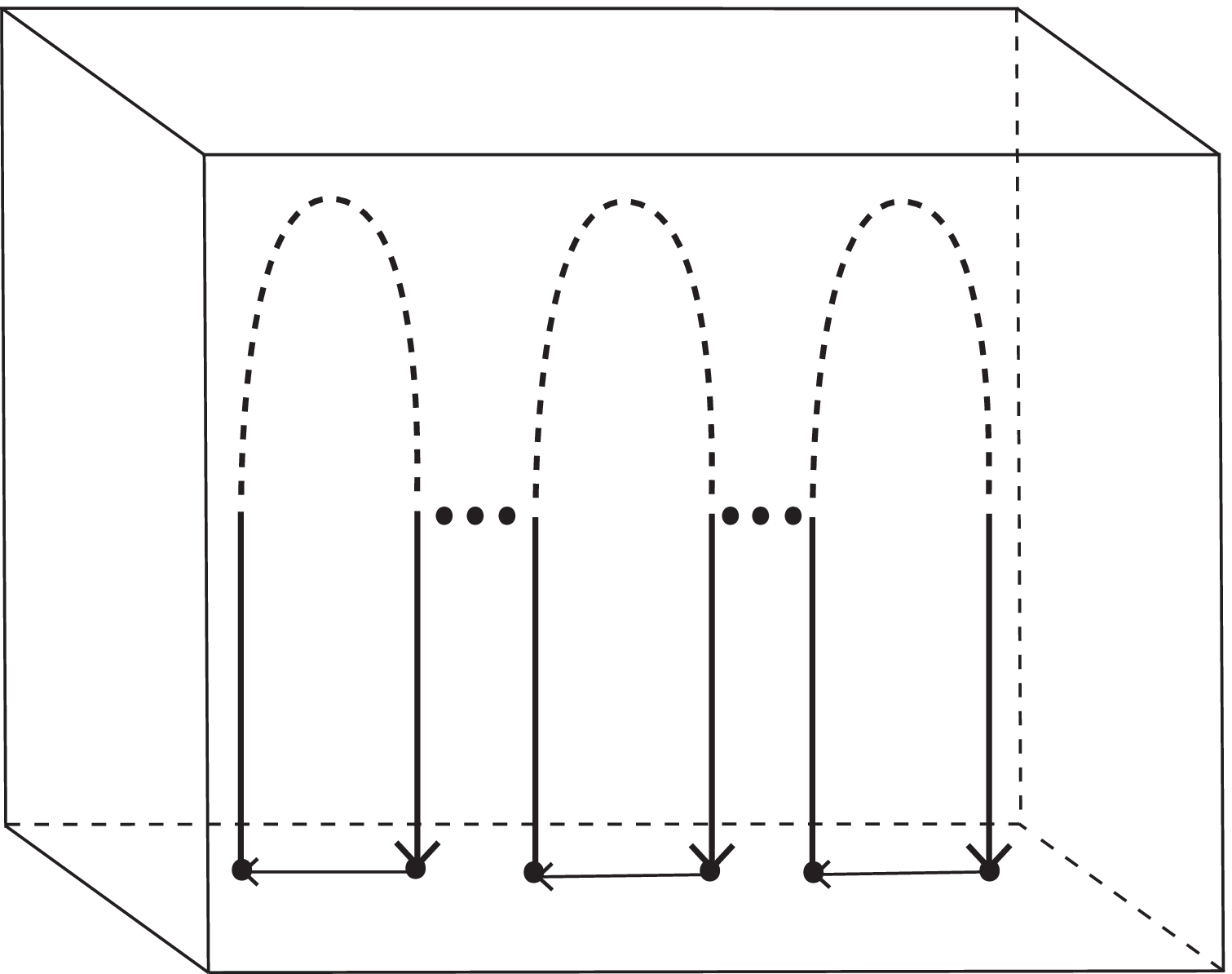}
   \put(35,5){$a_{1}$}
   \put(70,5){$a_{i}$}
   \put(100,5){$a_{n}$}
\end{overpic}
    \caption{}
    \label{closure}
  \end{center}
\end{figure}
We call $L(\gamma)$ the {\it closure} of $\gamma$.
On the other hand, for any oriented link $L$ in $S^{3}$, there is a bottom tangle $\gamma_{L}$ such that the closure of $\gamma_{L}$ is equal to $L$.
So we define the Milnor number of $L$ to be the Milnor number of $\gamma_{L}$.
Let $\Delta_{L}(I)$  be the greatest common devisor of $\mu_{L}(J)'s$, where $J$ is obtained from proper subsequence of $I$ by permuting cyclicly.
The {\it Milnor invariant $\omu_L(I)$} is the residue class of $\mu_L(I)$ modulo $\Delta_L(I)$.
We note that for a sequence $I$, if we have $\Delta_L(I)=0$, then the Milnor invariant $\omu_{L}(I)$ is equal to the Milnor number $\mu_{\gamma_{L}}(I)$.
Now we define the Milnor number of a clover link.

\begin{definition}
Let $c$ be an $n$-clover link and $F_{c}$ a disk/band surface of $c$.
Let $\gamma(F_{c})$ be the $n$-component bottom tangle obtained from $F_{c}$.
For a sequence $I$, {\it the Milnor number $\mu_{c}(I)$ of $c$} is defined to be the Milnor number $\mu_{\gamma(F_c)}(I)$.
\end{definition}

\begin{remark}\label{wada-inv}
While $\mu_{c}(I)$ depends on a choice of $F_{c}$, the first author~\cite{W} proved the following result:
Let $l_{c}$ be a link which is the disjoint union of leaves of $c$.
If the Milnor numbers of $l_c$ for sequences with length $\leq k$ vanish, then $\mu_{c}(I)$ is well-defined for any sequence $I$ with $|I|\leq 2k+1$.
\end{remark}

%%%%%%%%%%%%%%%%%%%%%%%%%%%%%%%%%%%%%%%%%%%%%%%%%%%%%%%%%%%%%%%%%%%%%%%%%%%%%%%%%%%%%%%%%%%%%%%%%%%%%%%%%%%%%%%%%%%%%%%%%%%%%%%%%%%%%%%%%%%%%%%%%%%%%%%%%%%%%%%%%%%%%%%%%%%%%%%%%%%%%%%%%%%%%%%%%%%%%%%%%%%%%%%%%%%%%%%%%%%%%%%%%%%%%%%%%%%%%%%%%%%%%%%%%%%%%%%%%%%%%%%%%%%%%%
\section{Proof of Theorem~\ref{mainthmmodulo}}

In this section we will give a proof of Theorem~\ref{mainthmmodulo}.

An $n$-component tangle $u=u_{1}\cup u_{2}\cup \cdots \cup u_{n}$ is an $n$-component {\it string link} if for each $i\ (=1,2,\ldots,n)$, the boundary $\partial u_{i}=\{ (\frac{2i-1}{2n+1},\frac{1}{2},0), (\frac{2i-1}{2n+1},\frac{1}{2},1)\} \subset \partial [0,1]^{3}$.
In particular, $u$ is {\it trivial} if for each $i\ (=1,2,\ldots,n)$, $u_{i}=\{(\frac{2i-1}{2n+1},\frac{1}{2})\}\times[0,1]$ in~$[0,1]^{3}$.

Here we introduce a {\em SL-move} \cite{W}  given by a string link $u$ which is a transformation of an $n$-component bottom tangle $\gamma=\gamma_{1}\cup\gamma_{2}\cup\cdots\cup\gamma_{n}$ with $\partial \gamma_{i}=\{ (\frac{2i-1}{2n+1},\frac{1}{2},0),(\frac{2i}{2n+1},\frac{1}{2},0) \}\subset\partial[0,1]^3$.

\begin{itemize}
\item[(1)]
Let $u=u_{1}\cup u_{2}\cup\cdots\cup u_{n}$ be an oriented ordered $n$-component string link in $[0,1]^{3}$.
For each $i\ (=1,2,\ldots,n)$, we consider an arc $u'_i$ which is parallel to the $i$th component $u_{i}$ of $u$ with opposite orientation and $\partial u'_{i}=\{ (\frac{2i}{2n+1},\frac{1}{2},0), (\frac{2i}{2n+1},\frac{1}{2},1) \}\subset\partial[0,1]^3$, see Figure~\ref{parallelstrands}. 

\begin{figure}[htbp]
    \begin{overpic}[width=5cm]{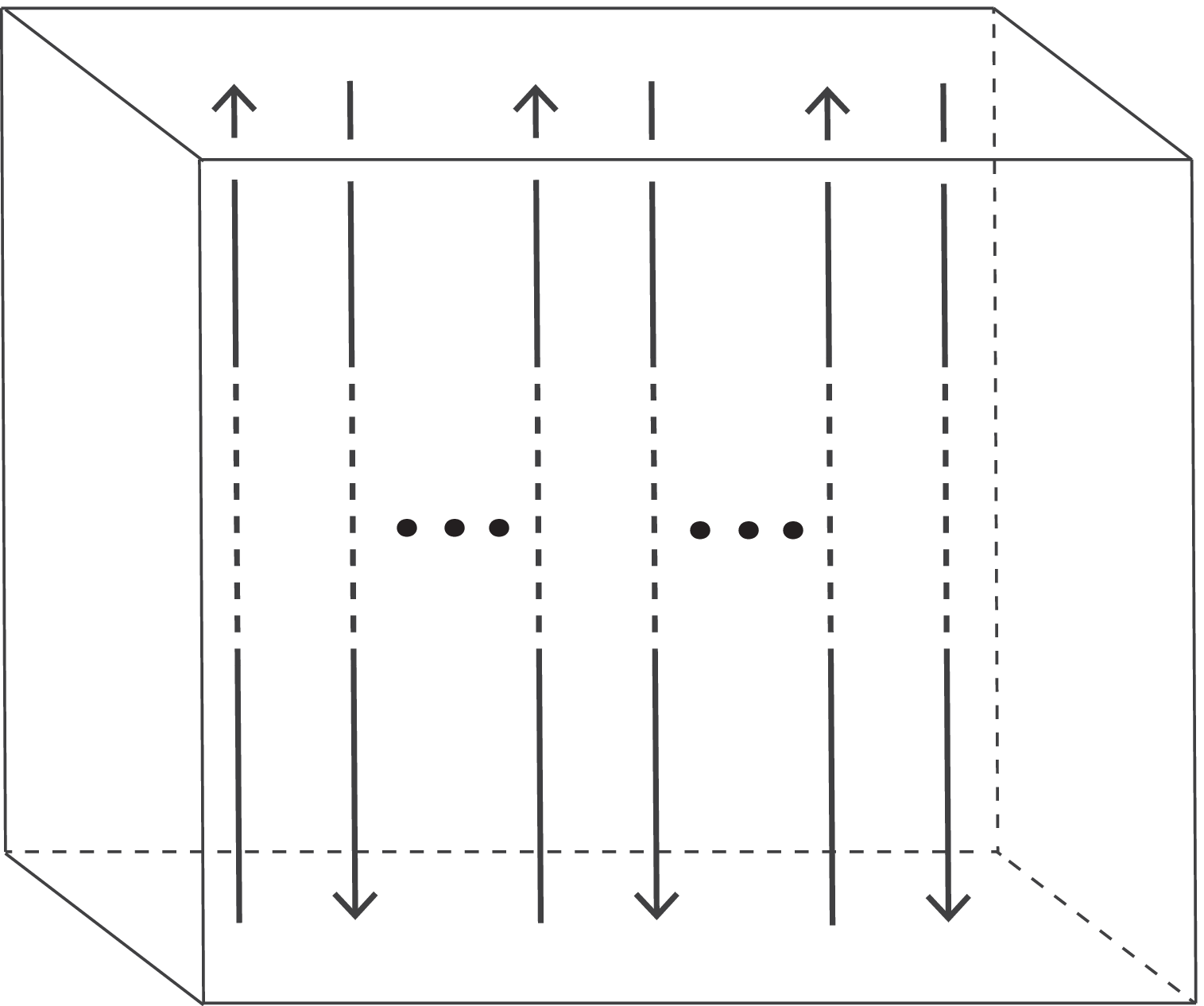}
    \put(23,111){$u_{1}$}
    \put(37,111){$u'_{1}$}
    \put(60,111){$u_{i}$}
    \put(74,111){$u'_{i}$}
    \put(93,111){$u_{n}$}
    \put(108,111){$u'_{n}$}
    \end{overpic}
    \caption{}
    \label{parallelstrands}
\end{figure}
\vspace{-0.5em}
\item[(2)] Let $\gamma'=\gamma'_{1}\cup\gamma'_{2}\cup\cdots\cup\gamma'_{n}$ be an $n$-component bottom tangle in $[0,1]^{3}$ defined by 
\begin{center}
$\gamma'_{i}=h_{0}(u_{i}\cup u'_{i})\cup h_{1}(\gamma_{i})$
\end{center}
for $i=1,2,\ldots,n$, where $h_{0},h_{1}:([0,1]\times[0,1])\times[0,1]\rightarrow ([0,1]\times[0,1])\times[0,1]$ are embeddings defined by 
\begin{center}
$h_{0}(x,t)=(x,\frac{1}{2}t)$ and $h_{1}(x,t)=(x,\frac{1}{2}+\frac{1}{2}t)$
\end{center}
for $x\in([0,1]\times[0,1])$ and $t\in[0,1]$.
\end{itemize}
\vspace{-0.5em}
\begin{figure}[htbp]
  \begin{center}
    \begin{overpic}[width=12cm]{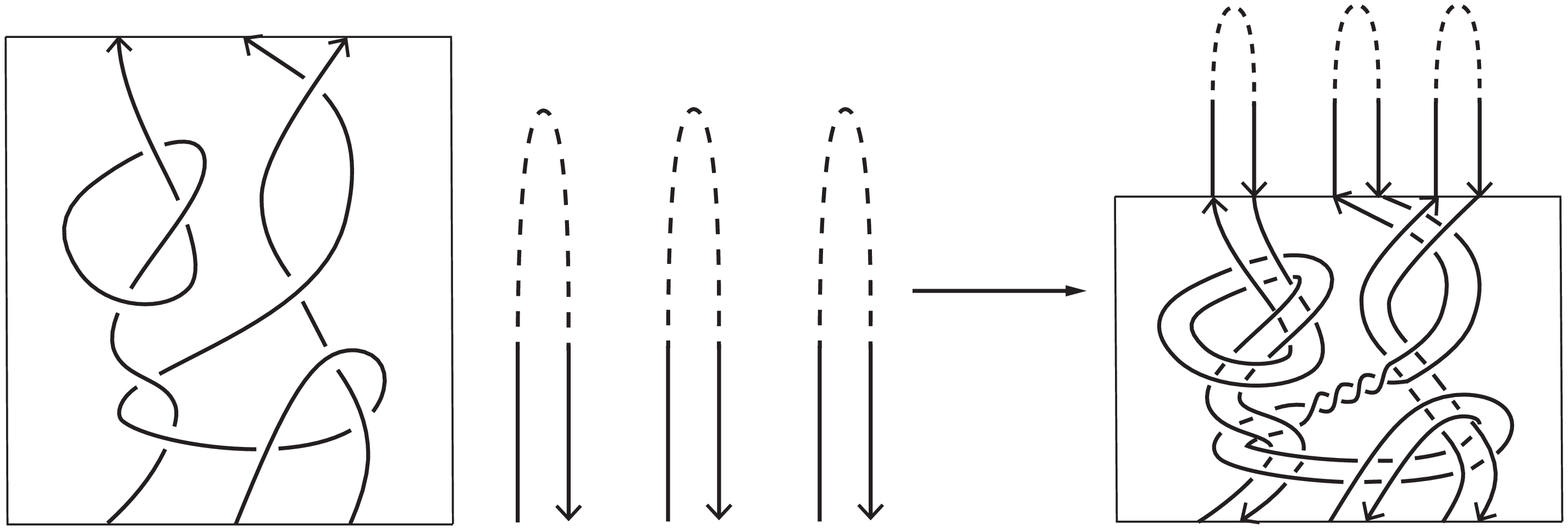}
    \put(45,-10){$u$}
    \put(290,-10){$\gamma'$}
    \put(149,-10){$\gamma$}
    \put(198,56){SL-move}
    \end{overpic}
    \caption{An example of a SL-move}
    \label{SL-move}
  \end{center}
\end{figure}

We say that {\it $\gamma'$ is obtained from $\gamma$ by a SL-move}. 
For example, see Figure~\ref{SL-move}. 
We note that if $u$ is trivial, a SL-move is just adding full-twists or nothing.
A SL-move is determined by a string link and a number of full-twists, 
that is, \lq SL' stands for String Link.

%%%%%%%%%%%%%%%%%%%%%%%%%%%%%%%%%%%%%%%%%%%%%%%%%%%%%%%%%%%%%%%%%%%%%%%%%%%%%%%%%%%%%%%%%%%%%%%%%%%%%%%%%%%%%%%%%%%%%%%%%%%%%%%%%%%%%%%%%%%%%%%%%%%%%%%%%%%%%%%%%%
\begin{proposition}
\label{propmodulo}
Let $\gamma$ be an $n$-component bottom tangle and $\gamma'$ a bottom tangle obtained from $\gamma$ by a SL-move.
If the Milnor numbers of $\gamma$ and $\gamma'$ for sequences with length $\leq k$ vanish,
then for any sequence $I$, 
\[\mu_{\gamma'}(I)\equiv\mu_{\gamma}(I)~\mod~\delta_{\gamma}^k(I),\]
where $\delta^k_{\gamma}(I)$ is the greatest common devisor of $\mu_{\gamma}(J)'s$ for 
a proper subsequence $J$ of $I$ which is obtained by removing at least $k+1$ indices.
\end{proposition}

\begin{proof}
The proof is by induction on the length $|I|=q$. 
If $q\leq k$, the proposition clearly holds. We note that, by the  induction hypothesis, 
$\delta^k_{\gamma}(I)=\delta^k_{\gamma'}(I)$. 
Denote respectively by $\alpha_i, \lambda_i$ (resp. $\alpha_i', \lambda_i'$) the $i$th meridian and $i$th longitude of $\gamma$ (resp. $\gamma'$) for $1\leq i\leq n$.
Let $E_X$ (resp. $E_Y$) be the Magnus expansion in non-commutative variables $X_1,\ldots ,X_n$ (resp. $Y_1,\ldots ,Y_n$) obtained by replacing $\alpha_i$ by $1+X_i$ (resp. $\alpha_i'$ by $1+Y_i$) for $1\leq i \leq n$.
Fix $j$,
by the assumption,
the Milnor numbers for $\gamma$ and $\gamma'$ of length $\leq k$ vanish,
so $E_X(\lambda_j)$ and $E_Y(\lambda'_j)$ can be written respectively in the form
\[E_X(\lambda_j)=1+F_j(X)~\text{and}~E_Y(\lambda'_j)=1+F'_j(Y),\]
where 
$F_j(X)(=F_j(X_1,...,X_n))$ and $F'_j(Y)(=F'_j(Y_1,...,Y_n))$  are terms 
of degree~$\geq~k$.

Here we define a set of polynomials;
\[D_j^k=\Bigl\{\sum\nu(i_1\ldots i_m)Y_{i_1}\cdots Y_{i_m}~\left|
\begin{array}{ll}
\nu(i_1\ldots i_m)\equiv0\mod\delta_{\gamma}^k(i_1\ldots i_mj),&m<q\\
\nu(i_1\ldots i_m)\in\mathbb{Z},&m\geq q
\end{array}
\right.
\Bigr\}.\]
Then it is enough to show
\[F'_j(Y)-F_j(Y)\in D_j^k.\]
The following claims are shown by similar to the assertions (16) and (18) in~\cite{M2}.
\begin{claim}
\label{claim1}
$D_j^k$ is a two-sided ideal of the formal power series ring in non-commutative variables 
$Y_1,\ldots,Y_n$ with integer coefficients.
\end{claim}
\begin{claim}
\label{claim2}
If at least $k$ variables are inserted anywhere in a term
$\mu(i_1i_2\ldots i_mj)Y_{i_1i_2\cdots i_m}$,
then the resulting term belongs to $D_j^k$.
\end{claim}

\begin{figure}[htbp]
  \begin{center}
    \begin{overpic}[width=10cm]{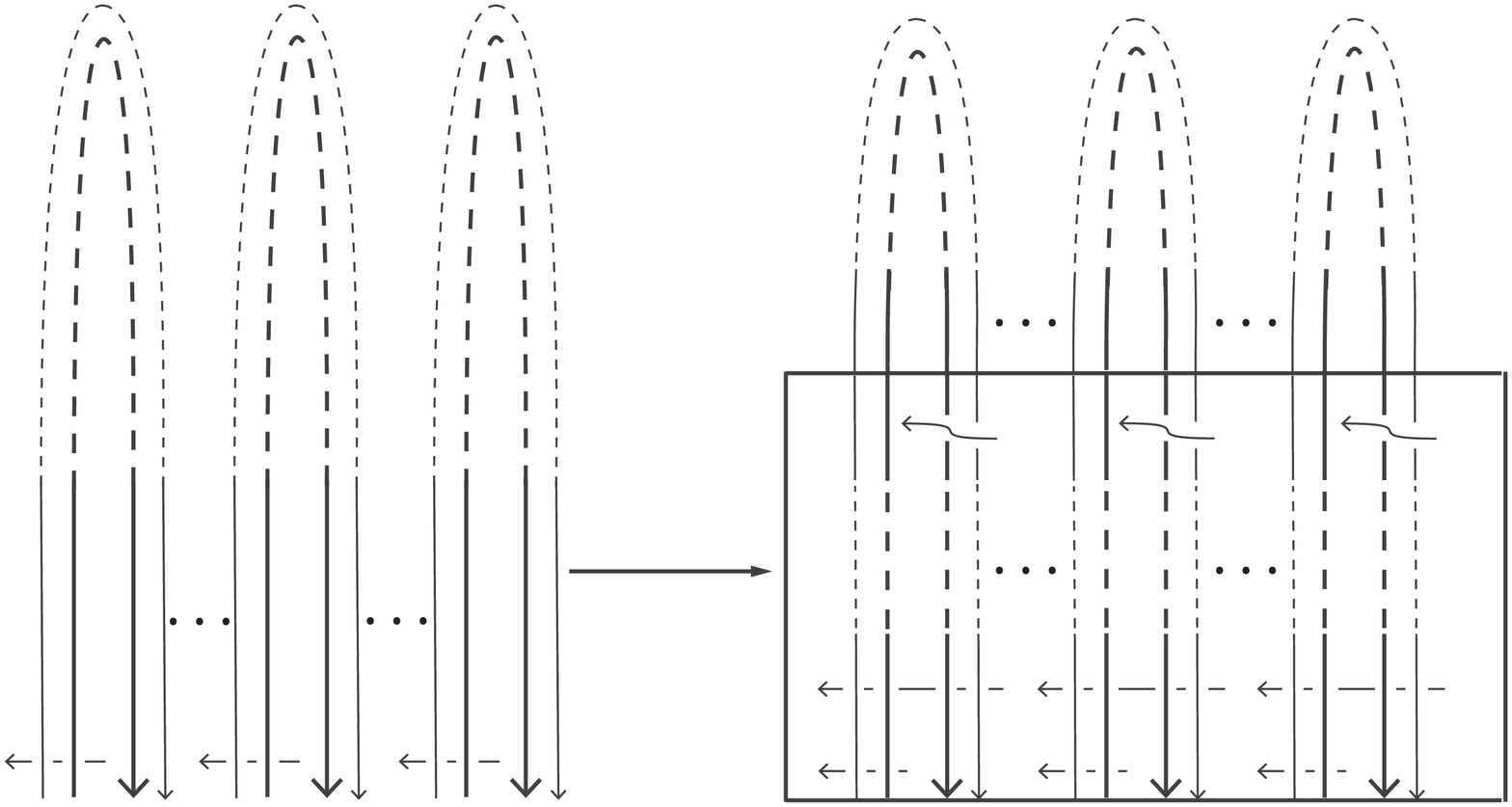}
    \linethickness{3pt}
    \put(53,-10){$\gamma$}
    \put(211,-10){$\gamma'$}
    \put(107,47){SL-move}
    \put(14,158){$\lambda_{1}$}
    \put(51,158){$\lambda_{i}$}
    \put(87,158){$\lambda_{n}$}
    \put(167,158){$\lambda'_{1}$}
    \put(210,158){$\lambda'_{i}$}
    \put(250,158){$\lambda'_{n}$}
    \put(-5,11){$\alpha_{1}$}
    \put(33,11){$\alpha_{i}$}
    \put(69,11){$\alpha_{n}$}
    \put(150,11){$\alpha'_{1}$}
    \put(192,11){$\alpha'_{i}$}
    \put(232,11){$\alpha'_{n}$}
    \put(190,66){$u_{1}$}
    \put(231,66){$u_{i}$}
    \put(272,66){$u_{n}$}
    \put(150,28){$\beta_{1}$}
    \put(192,28){$\beta_{i}$}
    \put(232,28){$\beta_{n}$}
    \end{overpic}
    \caption{}
    \label{bandbottomtangle}
  \end{center}
\end{figure}
Let $u_{i}$ be the $i$th longitude of a string link which gives the SL-move, see Figure~\ref{bandbottomtangle}.
In the proof of~\cite[Lemma~2.6]{W}, it is shown that 
the degree of each term in 
$E_Y(u_i^{\pm1})-1$ is at least $k+1$.
Set $E_Y(u_i)=1+G_i(Y)$ and $E_Y(u_i^{-1})=1+\overline{G_i}(Y)$, where $G_i(Y)$ and $\overline{G_i}(Y)$ mean the terms of degree $\geq k+1$.
Since $\alpha_i=u_i^{-1}\alpha'_iu_i$ (where $\alpha_i$ is assumed to be an element of $\pi_1([0,1]^3\setminus\gamma')$), 
we have
\[\begin{array}{rcl}
E_Y(\alpha_i)&=&E_Y(u_i^{-1}\alpha'_iu_i)\\
&=&(1+\overline{G_i}(Y))(1+Y_i)(1+G_i(Y))\\
&=&(1+\overline{G_i}(Y))(1+G_i(Y))+(1+\overline{G_i}(Y))Y_i(1+G_i(Y))\\
&=&1+Y_i+Y_iG_i(Y)+\overline{G_i}(Y)Y_i+\overline{G_i}(Y)Y_iG_i(Y).
\end{array}\]
Hence $E_Y(\lambda_j)$ is obtained from $E_X(\lambda_j)$ by substituting
$X_i$ for
\[Y_i+Y_iG_i(Y)+\overline{G_i}(Y)Y_i+\overline{G_i}(Y)Y_iG_i(Y).\]
Set $E_Y(\lambda_j)=1+H_j(Y)$, where $H_j(Y)$ is the terms of degree $\geq k$.
Note that terms of degree $\leq 2k$ of $H_j(Y)-F_j(Y)$ vanish, and that 
any term of $H_j(Y)-F_j(Y)$ of degree $\geq 2k+1$ is obtained from $F_j(Y)$ by inserting at least $k+1$ variables.
By Claim~\ref{claim2}, 
\[H_j(Y)-F_j(Y)\in D_j^k.\]

Since $\lambda'_j=u_j\lambda_ju_j^{-1}$ (where $\lambda_j$ is assumed to be an element of $\pi_1([0,1]^3\setminus\gamma')$),
\[\begin{array}{rcl}
E_Y(\lambda'_j)&=&E_Y(u_j\lambda_ju_j^{-1})\\
&=&(1+G_j(Y))(1+H_j(Y))(1+\overline{G_j}(Y))\\
&=&1+H_j(Y)+H_j(Y)\overline{G_j}(Y)+G_j(Y)H_j(Y)+G_j(Y)H_j(Y)\overline{G_j}(Y).
\end{array}\]
It follows from Claims~\ref{claim1} and \ref{claim2}  that we have 
\[F'_j(Y)-H_j(Y)=H_j(Y)\overline{G_j}(Y)+G_j(Y)H_j(Y)+G_j(Y)H_j(Y)\overline{G_j}(Y)\in D_j^k.\]
Since $H_j(Y)-F_j(Y)\in D_j^k$, by Claim~\ref{claim1}, we have
\[F'_j(Y)-F_j(Y)\in D_j^k.\]
This completes the proof.
\end{proof}
%%%%%%%%%%%%%%%%%%%%%%%%%%%%%%%%%%%%%%%%%%%%%%%%%%%%%%

\begin{proof}[Proof of Theorem~{\rm \ref{mainthmmodulo}}]
By Lemma ~\ref{lemmaPB-move}, any two disk/band surfaces $F_{c}$ and $F'_{c}$ of an $n$-component clover link $c$ are transformed into each other by the moves (a) and (b) illustrated in Figure~\ref{PB-move}.
So two bottom tangles $\gamma(F_c)$ and $\gamma(F'_c)$ are transformed into each other by a SL-move.
Since the both closures $L(\gamma(F_{c}))$ and $L(\gamma(F'_{c}))$ are ambient isotopic to $l_{c}$, 
by the hypothesis of Theorem~\ref{mainthmmodulo}, 
\[0=\omu_{l_{c}}(J)=\mu_{\gamma(F_{c})}(J)=\mu_{\gamma(F'_{c})}(J)\]
 for any sequence $J$ with $|J|\leq k.$
Hence by Proposition~\ref{propmodulo}, 
\[\mu_{\gamma(F'_{c})}(I)\equiv \mu_{\gamma(F_{c})}(I)~\mod~\delta_{\gamma({F_c})}^k(I)\]
 for any sequence $I$.
This completes the proof.
\end{proof}

%%%%%%%%%%%%%%%%%%%%%%%%%%%%%%%%%%%%%%%%%%%%%%%%%%%%%%
\section{
Proof of Theorems~\ref{mainthm2k+2} and \ref{thm4-clover}}

In this section we  will give proofs of Theorems~\ref{mainthm2k+2}, \ref{thm4-clover} and 
Corollary~\ref{EH-inv}. 

\begin{proposition}
\label{prop2k+2}
Let $\gamma$ be an $n$-component bottom tangle and $\gamma'$ a bottom tangle obtained from $\gamma$ by a SL-move which is given by a string link $u$.
If the Milnor numbers of $\gamma$ and $\gamma'$ for non-repeated sequences with length $\leq k$  vanish, then 
we have the following: 

\[\begin{array}{rcl}
&&\displaystyle\sum_{S\in \mathcal{S}_j^{2k+1}}(\mu_{\gamma'}(Sj)-\mu_{\gamma}(Sj))Y_S\\
&=&\displaystyle\sum_{\substack{|J|=|I|=k\\JIl\in \mathcal{S}_j^{2k+1}}}\mu_{\gamma}(Jj)\mu_{\gamma}(Il)
\sum_{i_s\in\{J\}}\mu_u(li_s)
Y_{J_{<s}}(
Y_{i_sIl}
-Y_{i_slI}
-Y_{Ili_s}
+Y_{lIi_s}
)Y_{J_{s<}}\\
&&~~~+\displaystyle\sum_{\substack{|J|=|I|=k\\JIl\in \mathcal{S}_j^{2k+1}}}\mu_{\gamma}(Jj)\mu_{\gamma}(Il)\mu_u(lj)(Y_{IlJ}-Y_{lIJ}-Y_{JIl}+Y_{JlI}).
\end{array}\]
\end{proposition}

\begin{proof}
We compare with the Magnus expansions of the $j$th longitudes of $\gamma$ and $\gamma'$.
Denote respectively by $\alpha_i, \lambda_i$ (resp. $\alpha_i', \lambda_i'$) the $i$th meridian and $i$th longitude of $\gamma$ (resp. $\gamma'$) for $1\leq i\leq n$.
Let $E_X$ (resp. $E_Y$) be the Magnus expansion in non-commutative variables $X_1,\ldots ,X_n$ (resp. $Y_1,\ldots ,Y_n$) obtained by replacing $\alpha_i$ by $1+X_i$ (resp. $\alpha_i'$ by $1+Y_i$) for $1\leq i \leq n$.
By the assumption,
the Minor numbers for $\gamma$ and $\gamma'$ of degree $k$ coincide,
hence denote by
\[E_X(\lambda_j)=1+\displaystyle\sum_{I\in \mathcal{S}_j^k}
{\mu_{\gamma}(Ij)}X_I+r_j(X)+\mathcal{O}_X(2)\]
and
\[E_Y(\lambda'_j)=1+\displaystyle\sum_{I\in \mathcal{S}_j^k}{\mu_{\gamma}(Ij)}Y_I+r'_j(Y)+\mathcal{O}_Y(2),\]
where $r_j(X)$ and $r'_j(Y)$ mean the terms of degree $\geq k+1$
and $\mathcal{O}_X(2)$ (resp. $\mathcal{O}_Y(2)$) denotes the terms which contain $X_i$ (resp. $Y_i$) at least $2$ times for some $i(=1,2,\ldots,n)$.
Let $f_j(X)=\displaystyle\sum_{I\in \mathcal{S}_j^k}{\mu_{\gamma}(Ij)}X_I$ and $f_j(Y)=\displaystyle\sum_{I\in \mathcal{S}_j^k}{\mu_{\gamma}(Ij)}Y_I$.

Let $u_{i}$ be the $i$th longitude of $u$
and let $\beta_{i}=[ \lambda_i', \alpha_i']={\lambda'_i}^{-1}{\alpha'_i}^{-1}\lambda'_i\alpha'_i$, see Figure~\ref{bandbottomtangle}.
Let $E_Z$ be the Magnus expansion in non-commutative variables $Z_1,\ldots ,Z_n$ obtained by replacing $\beta_i$ by $1+Z_i$ for $1\leq i \leq n$.
Then we have
\[E_Z(u_i)=1+\displaystyle\sum_{l\neq i}\mu_u(li)Z_l+\mathcal{O}(2).\]

First, we observe $E_Y(\beta_l)$ and $E_Y(\alpha_i)$, where $\alpha_{i}$ is assumed to be an element of $\pi_{1}([0,1]^{3}\setminus \gamma')$.
Since $E_Y(\lambda'_l)E_Y({\lambda'_l}^{-1})=1$, set
$E_Y({\lambda'_l}^{-1})=1-f_l(Y)+\overline{r'_l}(Y)+\mathcal{O}_Y(2)$,
where $\overline{r'_l}(Y)$ is the terms of degree $\geq k+1$.
Observe that
\[\begin{array}{rcl}
&&E_Y(\beta_l)\\
&=&E_Y({\lambda'_l}^{-1}{\alpha'_l}^{-1}\lambda'_l\alpha'_l)\\
&=&(1-f_l(Y)+\overline{r'_l}(Y)+\mathcal{O}_Y(2))(1-Y_l+\mathcal{O}_Y(2))(1+f_l(Y)+r'_l(Y)+\mathcal{O}_Y(2))(1+Y_l)\\
&=&1+f_l(Y)Y_l-Y_lf_l(Y)+\mathcal{O}(k+2)+\mathcal{O}_Y(2)\\
&=&1+\displaystyle\sum_{I\in \mathcal{S}_l^k}\mu_{\gamma}(Il)Y_IY_l
-Y_l\displaystyle\sum_{I\in \mathcal{S}_l^k}\mu_{\gamma}(Il)Y_I
+\mathcal{O}(k+2)+\mathcal{O}_Y(2)\\
&=&1+\displaystyle\sum_{I\in \mathcal{S}_l^k}\mu_{\gamma}(Il)(Y_{Il}-Y_{lI})+\mathcal{O}(k+2)+\mathcal{O}_Y(2).
\end{array}\]
This implies that
$E_Y(u_i)$ is obtained from $E_Z(u_i)$ by substituting $Z_l$ for
\[\displaystyle\sum_{I\in \mathcal{S}_l^k}\mu_{\gamma}(Il)(Y_{Il}-Y_{lI})+\mathcal{O}(k+2)+\mathcal{O}_Y(2).\]
So we have
\[E_Y(u_i)=1+\displaystyle\sum_{l\neq i}\mu_u(li)
\displaystyle\sum_{I\in \mathcal{S}_l^k}\mu_{\gamma}(Il)(Y_{Il}-Y_{lI})
+\mathcal{O}(k+2)+\mathcal{O}_Y(2).
\]
Let $g_i(Y)=\displaystyle\sum_{l\neq i}\mu_u(li)
\displaystyle\sum_{I\in \mathcal{S}_l^k}\mu_{\gamma}(Il)(Y_{Il}-Y_{lI})$.
Then we have 
\[E_Y(u_i^{-1})=1-g_i(Y)+\mathcal{O}(k+2)+\mathcal{O}_Y(2).\]
Since $\alpha_i=u_i^{-1}\alpha_i'u_i$,
we have
\[\begin{array}{l}
E_Y(\alpha_i)\\
=E_Y(u_i^{-1}\alpha'_iu_i)\\
=(1-g_i(Y)+\mathcal{O}(k+2)+\mathcal{O}_Y(2))(1+Y_i)(1+g_i(Y)+\mathcal{O}(k+2)+\mathcal{O}_Y(2))\\
=1+Y_i+Y_ig_i(Y)-g_i(Y)Y_i+\mathcal{O}(k+3)+\mathcal{O}_Y(2)\\
=1+Y_i+\displaystyle\sum_{l\neq i}\mu_u(li)
\displaystyle\sum_{I\in \mathcal{S}_l^k}\mu_{\gamma}(Il)(Y_{iIl}-Y_{ilI}-Y_{Ili}+Y_{lIi})+\mathcal{O}(k+3)+\mathcal{O}_Y(2).
\end{array}\]

Now we consider the difference $d_j(Y)=E_Y(\lambda_j)-(1+f_j(Y)+r_j(Y)+\mathcal{O}_Y(2))$.
Since $E_Y(\lambda_j)$ is obtained from $E_X(\lambda_j)(=1+f_j(X)+r_j(X)+\mathcal{O}_X(2))$
by substituting $X_{i}$ for 
\[Y_{i}+\displaystyle\sum_{l\neq i}\mu_u(li)\displaystyle\sum_{I\in \mathcal{S}_l^k}\mu_{\gamma}(Il)(Y_{iIl}-Y_{ilI}-Y_{Ili}+Y_{lIi})+\mathcal{O}(k+3)+\mathcal{O}_Y(2),\]
all terms of degree $\leq 2k$ of $d_j(Y)-\mathcal{O}_Y(2)$ vanish.
The terms of degree $2k+1$ in $d_j(Y)-\mathcal{O}_Y(2)$ is obtained from $f_j(Y)$ by substituting $Y_{i}$ for 
\[\displaystyle\sum_{l\neq i}\mu_u(li)
\displaystyle\sum_{I\in S_l^k}\mu_{\gamma}(Il)(Y_{iIl}-Y_{ilI}-Y_{Ili}+Y_{lIi})\]
for some $i\in\{1,2,\ldots,n\}$.
It follows that
\[\begin{array}{l}
d_j(Y)-(\mathcal{O}_Y(2)+\mathcal{O}(2k+2)+\mathcal{O}_{Y_j})\\
=\displaystyle\sum_{J\in \mathcal{S}_j^k}\mu_{\gamma}(Jj)
\displaystyle\sum_{i_s\in \{J\}}Y_{J_{<s}}\Bigl(\sum_{l\neq i_s}\mu_u(li_s)
\displaystyle\sum_{JIl\in \mathcal{S}_j^{2k+1}}\mu_{\gamma}(Il)(Y_{i_sIl}-Y_{i_slI}-Y_{Ili_s}+Y_{lIi_s})\Bigr)Y_{J_{s<}}\\

=\displaystyle
\sum_{\substack{|J|=|I|=k\\JIl\in \mathcal{S}_j^{2k+1}}}\mu_{\gamma}(Jj)\mu_{\gamma}(Il)
\displaystyle\sum_{i_s\in\{J\}}\mu_u(li_s)
Y_{J_{<s}}(Y_{i_sIl}-Y_{i_slI}-Y_{Ili_s}+Y_{lIi_s})Y_{J_{s<}},\\
\end{array}\]
where $\mathcal{O}_{Y_j}$ means the terms which contain $Y_j$ at least one time.

Finally, we observe the difference $E_Y(\lambda'_j)-(1+f_j(Y)+r_j(Y)+\mathcal{O}_Y(2))$.
Since $\lambda'_j=u_j\lambda_ju_j^{-1}$ (where $\lambda_{j}$ is assumed to be an element of $\pi_{1}([0,1]^{3}\setminus \gamma'$)),
we have
\[\begin{array}{l}
E_Y(\lambda'_j)\\
=E_Y(u_j\lambda_ju_j^{-1})\\
=1+(1+g_j(Y))(f_j(Y)+r_j(Y)+d_j(Y))(1-g_j(Y))
+\mathcal{O}(2k+2)+\mathcal{O}_Y(2)\\
=1+f_j(Y)+r_j(Y)+d_j(Y)+g_j(Y)f_j(Y)-f_j(Y)g_j(Y)
+\mathcal{O}(2k+2)+\mathcal{O}_Y(2).
\end{array}\]
So we have
\[\begin{array}{rcl}
&&E_Y(\lambda'_j)-(1+f_j(Y)+r_j(Y)+\mathcal{O}_Y(2))\\
&=&d_j(Y)
+\displaystyle\sum_{\substack{|J|=|I|=k\\JIl\in \mathcal{S}_j^{2k+1}}}
\mu_{\gamma}(Jj)\mu_{\gamma}(Il)\mu_u(lj)
(Y_{IlJ}-Y_{lIJ}-Y_{JIl}+Y_{JlI})\\
&&+\mathcal{O}(2k+2)+\mathcal{O}_Y(2)\\

&=&\displaystyle\sum_{\substack{|J|=|I|=k\\JIl\in \mathcal{S}_j^{2k+1}}}\mu_{\gamma}(Jj)\mu_{\gamma}(Il)
\displaystyle\sum_{i_s\in\{J\}}\mu_u(li_s)
Y_{J_{<s}}(
Y_{i_sIl}
-Y_{i_slI}
-Y_{Ili_s}
+Y_{lIi_s}
)Y_{J_{s<}}\\
&&+\displaystyle\sum_{\substack{|J|=|I|=k\\JIl\in \mathcal{S}_j^{2k+1}}}\mu_{\gamma}(Jj)\mu_{\gamma}(Il)\mu_u(lj)(Y_{IlJ}-Y_{lIJ}-Y_{JIl}+Y_{JlI})\\
&&+\mathcal{O}(2k+2)+\mathcal{O}_Y(2)+\mathcal{O}_{Y_j}.
\end{array}\]
This completes the proof.
\end{proof}
%%%%%%%%%%%%%%%%%%%%%%%%%%%%%%%%%%%%%%%%%%%%%%%%%%%%%%
\begin{proof}[Proof of Theorem~{\rm \ref{mainthm2k+2}}]
By Lemma ~\ref{lemmaPB-move}, any two disk/band surfaces $F_{c}$ and $F'_{c}$ of an $n$-component clover link $c$ are transformed into each other by the moves (a) and (b) in Figure~\ref{PB-move}.
So two bottom tangles $\gamma(F_c)$ and $\gamma(F'_c)$ are transformed into each other by a SL-move.
Since the both closures $L(\gamma(F_{c}))$ and $L(\gamma(F'_{c}))$ are ambient isotopic to $l_{c}$ and the hypothesis of Theorem~\ref{mainthm2k+2}, 
\[0=\omu_{l_{c}}(J)=\mu_{\gamma(F_{c})}(J)=\mu_{\gamma(F'_{c})}(J)\]
 for any sequence $J$ with $|J|\leq k.$ 
Since $\mu_u(pq)$ is the \lq linking number' of the $p$th component and the $q$th component of $u$, 
$\mu_u(pq)=\mu_u(qp)$ and
the set 
\[\{\mu_u(pq)~|~ u:\text{a string link}\}={\mathbb Z}\]
for any $p$ and $q$. 
This and Proposition~\ref{prop2k+2} give us the Thorem~\ref{mainthm2k+2}.
\end{proof}

%%%%%%%%%%%%%%%%%%%%%%%%%%%%%%%%%%%%%%%%%%%%%%%%%%%%%%
In order to prove Corollary~\ref{EH-inv} and Theorem~\ref{thm4-clover},
we need the following lemma given in~\cite{W}.
\begin{lemma}\cite[Lemma 4.1]{W}
\label{lemmaW}
Two $n$-clover links $c$ and $c'$ are edge-homotopic 
if and only if
there exist disk/band surfaces $F_c$ and $F_{c'}$ of $c$ and $c'$ respectively
such that the two bottom tangles $\gamma(F_c)$ and $\gamma(F_{c'})$ are link-homotopic.
\end{lemma}

\begin{proof}[Proof of Corollary~{\rm \ref{EH-inv}}]
Let $c$ and $c'$ be $n$-clover links.
We assume that they are edge-homotopic.
By Lemma~\ref{lemmaW},
there exist disk/band surfaces $F_c$ and $F_{c'}$ of $c$ and $c'$ respectively
such that $\gamma(F_c)$ and $\gamma(F_{c'})$ are link-homotopic.
This implies that $\mu_{\gamma(F_c)}(I)=\mu_{\gamma(F_{c'})}(I)$ for any non-repeated sequence $I$ 
\cite{M}, \cite{HL}. 
On the other hand, by Theorem~\ref{mainthm2k+2},  
the set $H_c(2k+2, j)$ (resp. $H_{c'}(2k+2, j)$) is obtained from 
the Milnor numbers of $\gamma(F_c)$ (resp. $\gamma(F_{c'})$) for  $F_c$ (resp. $F_{c'}$).
Hence we have  $H_c(2k+2, j)=H_{c'}(2k+2, j)$.
\end{proof}

%%%%%%%%%%%%%%%%%%%%%%%%%%%%%%%%%%%%%%%%%%%%%%%%%%%%%%
\begin{proof}[Proof of Theorem~{\rm \ref{thm4-clover}}]
Suppose that two $4$-clover links $c$ and $c'$ are edge-homotopic.
By Corollary~\ref{EH-inv}, $H_c(4, 4)=H_{c'}(4, 4)(\neq\emptyset)$.
By Lemma~\ref{lemmaW} there exist disk/band surfaces $F_c$ and $F_{c'}$
such that $\gamma(F_c)$ and $\gamma(F_{c'})$ are link-homotopic.
This implies that 
the Milnor numbers of $\gamma(F_c)$ and $\gamma(F_{c'})$ are equal for any non-repeated 
sequence~\cite{M},~\cite{HL}.
Since the Milnor numbers of length $\leq 3$ are always well-defined for clover links  
(see Remark~\ref{wada-inv}),
we have
$\mu_c(I)=\mu_{c'}(I)$
for any non-repeated sequence $I$ with $|I|\leq 3$. 

Conversely if $H_c(4, 4)\cap H_{c'}(4, 4)\neq\emptyset$,
then there exist disk/band surfaces $F_c$ and $F_{c'}$ of $c$ and $c'$ respectively such that
\[
\displaystyle
\sum_{S\in \mathcal{S}_4^3}\mu_{\gamma(F_c)}(S4)X_S
=\sum_{S\in \mathcal{S}_4^3}\mu_{\gamma(F_{c'})}(S4)X_S.
\]
In particular,
\[
\mu_{\gamma(F_c)}(1234)=\mu_{\gamma(F_{c'})}(1234)
~\text{and}~
\mu_{\gamma(F_c)}(2134)=\mu_{\gamma(F_{c'})}(2134).
\]
According to the link-homotopy classification theorem for string links  by N. Habegger and X. S. Lin \cite{HL}, 
for two 4-component string links (bottom tangles) that have common values of the Milnor numbers for 
non-repeated sequences with length $\leq 3$, they are link-homotopic if and only if 
their Milnor numbers for sequences $1234$ and $2134$ coincide, see also \cite[Theorem~4.3]{Y}.
This together with the hypothesis implies that
$\gamma(F_c)$ and $\gamma(F_{c'})$ are link-homotopic.
Therefore
$c$ and $c'$ are edge-homotopic by Lemma~\ref{lemmaW}.
This completes the proof.
\end{proof}

%%%%%%%%%%%%%%%%%%%%%%%%%%%%%%%%%%%%%%%%%%%%%%%%%%%%%%%%%%%%%%%%%%%%%%%%%%%%%%%%%%%%%%%%%%%%%%%%%%%%%%%%%%%%%%%%%%%%%%%%%%%%%%%%%%%%%%%%%%%%%%%%%%%%%%%%%%%%%%%%%%%%%%%%%%%%%%%%%%%%%%%%%%%%%%%%%%%%%%%%%%%%%%%%%%%%%%%%%%%%%%%%%%%%%%%%%%%%%%%%%%%%%%%%%%%%%%%%%%%%%%%%%%%%%%%%%%%%%%%%%%%%%%%%%%%%%%%%%%%%%%%%%%%%%%%%%%%%%%%%%%%%

\end{document}